\documentclass[11pt,british,refpage,intoc,bibliography=totoc,index=totoc,BCOR=7.5mm,captions=tableheading]{extarticle}
\usepackage{mathptmx}
\usepackage[T1]{fontenc}
\usepackage[latin9]{inputenc}
\usepackage[a4paper]{geometry}
\usepackage{color}
\usepackage{babel}
\usepackage{amsmath}
\usepackage{amsthm}
\usepackage{amssymb}
\usepackage{graphicx}
\usepackage[unicode=true,pdfusetitle,
 bookmarks=true,bookmarksnumbered=true,bookmarksopen=false,
 breaklinks=false,pdfborder={0 0 1},backref=false,colorlinks=true]
 {hyperref}
\hypersetup{
 pdfborderstyle=,linkcolor=black,citecolor=blue,urlcolor=black,filecolor=blue,pdfpagelayout=OneColumn,pdfnewwindow=true,pdfstartview=XYZ,plainpages=false}

\makeatletter
\theoremstyle{plain}
\newtheorem{thm}{\protect\theoremname}
\theoremstyle{remark}
\newtheorem*{rem*}{\protect\remarkname}
\theoremstyle{remark}
\newtheorem{rem}[thm]{\protect\remarkname}
\theoremstyle{plain}
\newtheorem{cor}[thm]{\protect\corollaryname}
\theoremstyle{definition}
\newtheorem{example}[thm]{\protect\examplename}

\@ifundefined{date}{}{\date{}}
\usepackage{babel}

\usepackage{amsfonts}
\usepackage{dsfont}

\allowdisplaybreaks[4]
\DeclareMathAlphabet{\mathcal}{OMS}{cmsy}{m}{n}

  \providecommand{\corollaryname}{Corollary}
  
  \providecommand{\remarkname}{Remark}
\providecommand{\theoremname}{Theorem}

  \providecommand{\corollaryname}{Corollary}
  
  \providecommand{\theoremname}{Theorem}
  
  \providecommand{\examplename}{Example}
  \providecommand{\remarkname}{Remark}

\newcommand{\lt}{\left}
\newcommand{\rt}{\right}

\providecommand{\corollaryname}{Corollary}
\providecommand{\examplename}{Example}
\providecommand{\remarkname}{Remark}
\providecommand{\theoremname}{Theorem}

\providecommand{\corollaryname}{Corollary}
\providecommand{\examplename}{Example}
\providecommand{\remarkname}{Remark}
\providecommand{\theoremname}{Theorem}

\providecommand{\corollaryname}{Corollary}
\providecommand{\examplename}{Example}
\providecommand{\remarkname}{Remark}
\providecommand{\theoremname}{Theorem}

\makeatother

\providecommand{\corollaryname}{Corollary}
\providecommand{\examplename}{Example}
\providecommand{\remarkname}{Remark}
\providecommand{\theoremname}{Theorem}

\begin{document}
\title{Explicit solutions for a class of nonlinear backward stochastic differential
equations and their nodal sets }
\author{Zengjing Chen\thanks{Department of Mathematics, Shandong University, Jinan, China. Email:
zjchen@sdu.edu.cn}\ , \ Shuhui Liu\thanks{Department of Mathematics, Shandong University, Jinan, China. Email:
shuhuiliusdu@gmail.com}\ , \ Zhongmin Qian\thanks{Mathematical Institute, University of Oxford, United Kingdom. Email:
qianz@maths.ox.ac.uk}\ , \ Xingcheng Xu\thanks{School of Mathematical Sciences, Peking University, Beijing, China.
Email: xuxingcheng@pku.edu.cn} }
\maketitle
\begin{abstract}
In this paper, we investigate a class of nonlinear backward stochastic
differential equations (BSDEs) arising from financial economics, and
give specific information about the nodal sets of the related solutions.
As applications, we are able to obtain the explicit solutions to an
interesting class of nonlinear BSDEs including the $k$-ignorance
BSDE arising from the modeling of ambiguity of asset pricing.
\end{abstract}
\vspace{5mm}
 \hspace{11mm}\textbf{Keywords:} Cameron-Martin formula, Feynman-Kac
formula, nodal set, nonlinear BSDE, 

\hspace{5mm}parabolic equation \vspace{5mm}

\hspace{5mm}\textbf{MSC(2010):} 60G05, 60G17, 60H10, 60H30 

\section{Introduction}

In a seminal paper \cite{peng1}, Pardoux and Peng (1990) studied
a non-linear backward stochastic differential equation (BSDE)
\begin{equation}
dY_{t}=-g(t,Y_{t},Z_{t})dt+Z_{t}dB_{t},\quad Y_{T}=\xi,\label{1.1}
\end{equation}
where $B$ is a Brownian motion on a probability space $(\varOmega,\mathcal{F},P)$,
$T>0$, and $\xi$ is measurable with respect to Brownian motion trajectories
up to $T$. These authors proved, under some assumptions on the non-linear
driver $g$ and the terminal value $\xi$, that BSDE (\ref{1.1})
possesses a unique solution, a pair of adapted processes $Y$ and
$Z$ satisfying stochastic integral equation
\[
Y_{t}=\xi+\int_{t}^{T}g(s,Y_{s},Z_{s})ds-\int_{t}^{T}Z_{s}dB_{s}
\]
for $0\leq t\leq T$. In past two decades, many researchers have worked
on the theory of BSDEs and have obtained many excellent results about
the solution pair $(Y_{t},Z_{t})$. Since the publication of \cite{peng1},
the theory of BSDEs has been applied to mathematical finance, stochastic
control, partial differential equations, stochastic game and so on,
see for example \cite{chen2,r,MPY,par,peng2,peng0,peng4} and the
literature therein. Explicit solutions to (\ref{1.1}) are known only
in few cases, mainly for the case where $g(t,y,z)$ is linear in $y$
and $z$. It is easy to see that the solution to a linear BSDE is
given by Feynman-Kac's formula (see for example Peng \cite{peng1}).
For a non-linear driver $g(t,y,z)$, little is known about $(Y_{t},Z_{t})$
due to lack of an explicit formula, but see \cite{chen1}, in which
Chen et al. have obtained an interesting co-monotonic theorem of $(Z_{t})$
for a non-linear but special driver $g(t,y,z)$. It remains a challenging
problem in general to derive useful information about solutions of
BSDEs.

In applications to some problems in financial economics, it is useful
to have an explicit expression for the solution of BSDE (\ref{1.1}).
For models appearing in mathematical finance, one needs to determine
the signs of solutions $(Z_{t})$, which allow to identify the monotone
ranges of active hedging. Therefore researchers are very interested
in determining the zeros of $(Z_{t})$ for such BSDE models, i.e.
the nodal set of the process $(Z_{t})$.

The goal of this paper is to identify the nodal sets of solutions
$(Z_{t})$ to a class of non-linear BSDEs which arise from financial
economics, and to identify the monotone ranges of $(Z_{t})$ accordingly.
Our results will cover the so-called $k$-ignorance model in continuous
recursive utilities, studied by Chen and Epstein \cite{chen2}. The
model is a simple BSDE:
\begin{equation}
Y_{t}=\xi+\int_{t}^{T}k|Z_{s}|ds-\int_{t}^{T}Z_{s}dB_{s}\label{1.2}
\end{equation}
for $0\leq t\leq T$, where $k>0$ is a model parameter. (\ref{1.2})
is perhaps the simplest non-linear BSDE. It has significant applications
in discussing non-linear risk measures. Chen et al. \cite{chen3,chen1}
have shown that if $\xi=\varphi(B_{T})$ and $\varphi$ is monotonic,
then the solution $(Y,Z)$ of (\ref{1.2}) can be computed explicitly.
In this case, Chen et al. \cite{chen3,chen1} observed that (\ref{1.2})
can be reduced to an equivalent linear BSDE, so that an explicit formula
may be obtained accordingly. If $\varphi$ is not monotonic, it remains
open to solve BSDE (\ref{1.2}) explicitly. By exploring the information
on the nodal set of $(Z_{t})$, we are able to work out explicit solutions
for a class of non-linear BSDEs, including the $k$-ignorance models,
where $\varphi$ is not necessary monotone. As an application, we
therefore are able to give an explicit representation of the solution
$(Y_{t},Z_{t})$ for the $k$-ignorance model (\ref{1.2}), where
the terminal value is Markovian and $\varphi(x)=x^{2}$ or $\varphi(x)=I_{[a,b)}(x)$.
We should point out that the $k$-ignorance model (\ref{1.2}) with
these terminal values plays an important role in modeling ambiguity
of asset pricing, and we will discuss this point in the last part
of the article. For this aspect, the reader should also refer to Chen
and Epstein \cite{chen2} and the literature therein too.

The paper proceeds as follows. In Section \ref{section-2} we first
introduce some notions, notations and a few basic facts about BSDEs,
which will be used through the paper. We then prove the main results
of the paper, i.e. identifying the nodal set of the solution $(Z_{t})$
to (\ref{1.1}) under some assumptions on its driver $g$ and its
terminal $\xi$. In section \ref{section 3} we give an explicit formula
for the $k$-ignorance model with a suitable terminal value. In section
\ref{section-4}, by applying our general result about the sign of
$(Z_{t})$, we work out the explicit solutions to several examples
where $\xi=I_{[a,b)}(B_{T})$ or $\xi=B_{T}^{2}$, and the driver
$g(z)=k|z|$. We conclude the paper in section \ref{section 5} by
discussing an application of our results in robust pricing in an incomplete
market.

\section{The main results}

\label{section-2}

Let us begin with the notion of backward stochastic differential equations,
recall the basic result on BSDEs and establish notations we will use
in what follows. Let $(B_{t})_{t\geq0}$ be a standard one dimensional
Brownian motion on a probability space $(\Omega,\mathcal{F},P)$.
Let $(\mathcal{F}_{t})$ be the $\sigma$-filtration generated by
the Brownian motion, that is, $\mathcal{F}_{t}=\sigma\{B_{s};0\leq s\leq t\}$
for $t\geq0$.

The driver in formulating the BSDE to be studied in this paper is
a deterministic real function $g(t,y,z)$ for $(t,y,z)\in[0,T]\times\mathbb{R}\times\mathbb{R}$,
which satisfies the following conditions:

\bigskip{}

\emph{(A.1) Lipschitz condition}. There exists a constant $k\geq0$,
such that
\begin{equation}
|g(t,y_{1},z_{1})-g(t,y_{2},z_{2})|\leq k(|y_{1}-y_{2}|+|z_{1}-z_{2}|)\label{eq:A.1}
\end{equation}
for all $t\geq0$, $y_{1},y_{2}\in\mathbb{R}$ and $z_{1},z_{2}\in\mathbb{R}$;
and

\emph{(A.2) Normalization condition}. $g(t,y,0)=0$ for any $(t,y)\in[0,T]\times\mathbb{R}$.

\bigskip{}

We will use the standard notation that $L^{2}(\Omega,\mathcal{F}_{t},P)$
denote the space of $\mathcal{F}_{t}$-measurable and square integrable
random variables on $(\Omega,\mathcal{F},P)$ for each $t\geq0$.
Let
\[
\mathcal{M}(0,T,\mathbb{R})=\left\{ (v_{t})_{t\in[0,T]}:\textrm{ real valued}(\mathcal{F}_{t})\text{-adapted process with }E\left[\int_{0}^{T}|v_{t}|^{2}dt\right]<\infty\right\} .
\]

The fundamental result obtained in Pardoux-Peng \cite{peng1} is the
following. If $g$ satisfies \textit{(A.1), (A.2)}, and $\xi\in L^{2}(\Omega,\mathcal{F}_{T},P)$,
BSDE (\ref{1.1}) has a unique solution, i.e., there is a pair of
adapted processes $Y,Z\in\mathcal{M}(0,T,\mathbb{R})$, which solve
(\ref{1.1}) in the sense that
\begin{equation}
Y_{t}=\xi+\int_{t}^{T}g(s,Y_{s},Z_{s})ds-\int_{t}^{T}Z_{s}dB_{s}\label{eq:bsde-q1}
\end{equation}
for all $t\in[0,T]$.

We are interested in Markovian case, that is, the terminal value $\xi$
in (\ref{eq:bsde-q1}) depends only on $B_{T}$, that is, $\xi=\varphi(B_{T})$,
so that
\begin{equation}
Y_{t}=\varphi(B_{T})+\int_{t}^{T}g(s,Y_{s},Z_{s})ds-\int_{t}^{T}Z_{s}dB_{s}.\label{3.1}
\end{equation}

Let us isolate the following assumptions on $\varphi$, which will
be used in our main results.

\bigskip{}

(\emph{H.1}) There is $c\in\mathbb{R}$, $\varphi$ is symmetric about
$c$, that is , $\varphi(c-x)=\varphi(c+x)$ for all $x\in\mathbb{R}$.

(\emph{H.2}) $\varphi$ is monotone on $[c,\infty)$.

\bigskip{}

We are now in a position to state our first result of the paper.
\begin{thm}
\label{lemma1} Let $g\in C_{b}^{1,3}(\mathbb{R}_{+}\times\mathbb{R}\times\mathbb{R})$
satisfying (A.1) and (A.2), and $\varphi\in C^{3}(\mathbb{R})$. Assume
that the derivatives $\varphi^{(i)}$ (where $i=0,1,2,3$) have at
most polynomial growth.

(1) Let $u(t,x)$ be the unique solution of Cauchy's initial problem
of the parabolic equation
\begin{equation}
\begin{cases}
 & \partial_{t}u=\frac{1}{2}\partial_{xx}^{2}u+g(t,u,\partial_{x}u),\textrm{ in }(0,\infty)\times\mathbb{R},\\
 & u(0,x)=\varphi(x).
\end{cases}\label{pde1}
\end{equation}
Then $Y_{t}=u(T-t,B_{t})$ and $Z_{t}=\partial_{x}u(T-t,B_{t})$ are
the unique solution pair of BSDE (\ref{3.1}).

(2) If in addition $\varphi$ satisfies (H.1) and (H.2), and $g(t,y,z)=g(t,y,-z)$
for any $t\in[0,T]$ and $y,z\in\mathbb{R}$, then

(i) $\partial_{x}u(t,c)=0$ for every $t>0$.

(ii) $w(t,x)=\partial_{x}u(t,x)$ is the unique solution to the initial
value problem of the parabolic equation
\begin{equation}
\partial_{t}w=\frac{1}{2}\partial_{xx}^{2}w+\partial_{z}g(t,u,w)\cdot\partial_{x}w+\partial_{y}g(t,u,w)\cdot w\label{pde3}
\end{equation}
and
\begin{equation}
w(0,x)=\varphi'(x),\quad\textrm{ for }x\in\mathbb{R}.\label{eq:w-int-a1}
\end{equation}
Moreover $w(t,c)=0$ for $t\geq0$.

(iii) Let $x\in\mathbb{R}$ and $0\leq t\leq T$. Let
\[
a_{s,t}=\partial_{y}g\left(t-s,u(t-s,X_{s}^{x}),w(t-s,X_{s}^{x})\right)\quad\textrm{, }b_{s,t}=\partial_{z}g\left(t-s,u(t-s,X_{s}^{x}),w(t-s,X_{s}^{x})\right)
\]
for $0\leq s\leq t$, where $X_{s}^{x}=x+B_{s}$. Define the stochastic
exponential martingale
\begin{equation}
N_{s}=\exp\left\{ \int_{0}^{s}b_{r,t}dB_{r}-\frac{1}{2}\int_{0}^{s}b_{r,t}^{2}dr\right\} \label{eq:c-m-q1}
\end{equation}
for $0\leq s\leq t$. Then
\begin{equation}
w(t,x)=E\left[N_{t}\varphi'(X_{t}^{x})\cdot e^{\int_{0}^{t}a_{s,t}ds}1_{\left\{ t<\tau\right\} }\right]\label{key-req1}
\end{equation}
for every $t\geq0$, where $\tau=\inf\left\{ s\geq0,\ X_{s}^{x}=c\right\} $.
\end{thm}

\begin{proof}
Since $\varphi$ is a $C^{3}$-function with polynomial growth and
$g\in C^{1,3}$, so by the theory of parabolic equations of second
order, (\ref{pde1}) possesses a unique solution $u(t,x)$ which belongs
to $C^{1,3}([0,T]\times\mathbb{R})$, see for example \cite{Friedman 1964}.

By applying It?s formula, $Y_{t}=u(T-t,B_{t})$, $Z_{t}=\partial_{x}u(T-t,B_{t})$
solve BSDE (\ref{3.1}), and the conclusion follows from the uniqueness
of the solution to BSDE (\ref{3.1}), which proves the first claim.


Now we prove (2). Since $g(t,y,\cdot)$ is symmetric about $0$, one
can verify that $u(t,c-x)$ and $u(t,x+c)$ are solutions to the parabolic
equation
\begin{equation}
\partial_{t}v=\frac{1}{2}\partial_{xx}^{2}v+g(t,v,\partial_{x}v),\textrm{ in }(0,\infty)\times\mathbb{R}\label{eq:par-qq1}
\end{equation}
and $v(0,x)$ coincides with $\varphi(c-x)$ and $\varphi(c+x)$ respectively.
Since $\varphi$ is symmetric about $c$, that is $\varphi(c-x)=\varphi(c+x)$,
by the uniqueness of the initial value problem for the parabolic equation
(\ref{eq:par-qq1}) we may conclude that $u(t,c+x)=u(t,c-x)$. It
in turn yields that $\partial_{x}u(t,c+x)=-\partial_{x}u(t,c-x)$
for $(t,x)\in\mathbb{R}_{+}\times\mathbb{R}$. In particular $\partial_{x}u(t,c)=0$
for every $t\geq0$. We have thus proven (i).

(ii) follows immediately by differentiating (\ref{pde1}) in $x$.

(iii) Under assumptions on $g(t,y,z)$, $a_{s,t}$ and $b_{s,t}$
are bounded processes where $0\leq s\leq t\leq T$. Since $\varphi\in C^{3}$
and $\varphi^{(i)}$ (where $i=0,1,2,3)$ possess at most polynomial
growth, the unique strong solution $u(t,x)$ to the problem (\ref{pde1})
belongs to $C^{1,3}([0,T]\times\mathbb{R})$. In particular we have
$w\in C^{1,2}([0,T]\times\mathbb{R})$.

Let us first consider the case where $\varphi^{(i)}$ (where $i=0,1,2,3$)
are bounded. For this case, the second order derivative of $u(t,x)$,
that is, $\partial_{x}w(t,x)$ are bounded in $[0,T]\times\mathbb{R}$.
Let $0\leq t\leq T$ be any but fixed.

Define $q(s)$ by solving the ordinary differential equation: $dq(s)=a_{s,t}q(s)ds$
with $q(0)=1$. Then $q(s)$ has finite variations, and is a bounded
process. $N_{s}$ is the solution to the exponential martingale equation:
$dN_{s}=N_{s}b_{s,t}dB_{s}$ and $N_{0}=1$. Of course $N$ is just
the stochastic exponential of $\int_{0}^{\cdot}b_{r,t}dB_{r}$, where
$(b_{s,t})_{s\leq t}$ is a bounded process (while its bound may depend
on $t$). Let $M_{s}=q(s)N_{s}w(t-s,X_{s}^{x})$, where $0\leq s\leq t$.

By It?s formula we have
\begin{align*}
dM_{s} & =q(s)d\left[N_{s}w(t-s,X_{s}^{x})\right]+N_{s}w(t-s,X_{s}^{x})a_{s,t}q(s)ds\\
 & =q(s)N_{s}b_{s,t}w(t-s,X_{s}^{x})dB_{s}+q(s)N_{s}dw(t-s,X_{s}^{x})\\
 & +q(s)N_{s}b_{s,t}\partial_{x}w(t-s,X_{s}^{x})ds+q(s)N_{s}a_{s,t}w(t-s,X_{s}^{x})ds.
\end{align*}
Since $w$ solves (\ref{pde3}), so that
\begin{align*}
dw(t-s,X_{s}^{x}) & =\left(-\partial_{s}w(t-s,X_{s}^{x})+\frac{1}{2}\partial_{xx}^{2}w(t-s,X_{s}^{x})\right)ds+\partial_{x}w(t-s,X_{s}^{x})dB_{s}\\
 & =\left(-a_{s,t}w(t-s,X_{s}^{x})-b_{s,t}\partial_{x}w(t-s,X_{s}^{x})\right)ds+\partial_{x}w(t-s,X_{s}^{x})dB_{s}.
\end{align*}
Substituting this into the previous equality for $M$, we obtain that
\[
dM_{s}=q(s)N_{s}\left[\partial_{x}w(t-s,X_{s}^{x})+b_{s,t}w(t-s,X_{s}^{x})\right]dB_{s}.
\]
We claim that $M$ is a square integrable martingale. In fact, since
\[
\left|q(s)N_{s}\left[\partial_{x}w(t-s,X_{s}^{x})+b_{s,t}w(t-s,X_{s}^{x})\right]\right|\leq C_{1}N_{s}
\]
for some positive constant $C_{1}$ depending on $t$, but not on
$s\leq t.$

But
\begin{align*}
E\left[|N_{s}|^{2}\right] & =E\left[\exp\left(2\int_{0}^{s}b_{r,t}dB_{r}-\int_{0}^{s}|b_{r,t}|^{2}dr\right)\right]\\
 & \leq C_{2}E\left[\exp\left(2\int_{0}^{s}b_{r,t}dB_{r}-2\int_{0}^{t}|b_{r,t}|^{2}dr\right)\right]\\
 & =C_{2}<\infty,
\end{align*}
where $C_{2}$ is a positive constant. Therefore,
\begin{align*}
E\left[|M_{t}|^{2}\right] & =E\left(M_{0}+\int_{0}^{t}\left[q(s)N_{s}\partial_{x}w(t-s,X_{s}^{x})+b_{s,t}w(t-s,X_{s}^{x})\right]dB_{s}\right)^{2}\\
 & \leq2E(M_{0}^{2})+2E\left(\int_{0}^{t}\left[q(s)N_{s}\partial_{x}w(t-s,X_{s}^{x})+b_{s,t}w(t-s,X_{s}^{x})\right]dB_{s}\right)^{2}\\
 & =C_{0}+2E\left(\int_{0}^{t}\left[q(s)N_{s}\partial_{x}w(t-s,X_{s}^{x})+b_{s,t}w(t-s,X_{s}^{x})\right]^{2}ds\right)\\
 & \leq C_{0}+2C_{1}^{2}E\left(\int_{0}^{t}N_{s}^{2}ds\right)\\
 & =C_{0}+2C_{1}^{2}\int_{0}^{t}E(N_{s}^{2})ds\\
 & \leq C_{0}+2C_{1}^{2}\cdot C_{2}t\\
 & <\infty
\end{align*}
which implies that $(M_{s})$ is a square integrable martingale up
to time $t$.

Since
\[
\tau=\inf\left\{ s\geq0,\ X_{s}^{x}=c\right\} =\inf\left\{ s\geq0,\ B_{s}=c-x\right\}
\]
is a stopping time, finite almost surely, see (2.6) in \cite{shreve},
by stopping theorem for martingales, we have $E\left(M_{0}\right)=E\left(M_{t\wedge\tau}\right)$
, which implies that, since $w(s,c)=0$ for all $s\geq0$,
\begin{align*}
w(t,x) & =E\left(q(t\wedge\tau)N_{t\wedge\tau}w(t-t\wedge\tau,X_{t\wedge\tau}^{x})\right)\\
 & =E\left[N_{t}\varphi'(X_{t}^{x})e^{\int_{0}^{t}a_{r,t}dr}1_{\left\{ t<\tau\right\} }\right]+E\left[q(\tau)N_{\tau}w(t-\tau,X_{\tau}^{x})1_{\left\{ \tau\leq t\right\} }\right]\\
 & =E\left[N_{t}\varphi'(X_{t}^{x})e^{\int_{0}^{t}a_{r,t}dr}1_{\{t<\tau\}}\right].
\end{align*}
A simple approximation procedure allow us to validate the representation
for the case where $\varphi^{(i)}$ (where $i=0,1,2,3)$ possess polynomial
growth. The proof is complete.
\end{proof}
\begin{rem*}
The representation (\ref{key-req1}) is basically Feynman-Kac's formula
for the stopped Brownian motion (killed at hitting the level $c$),
together with the Cameron-Martin formula, see for more information
Pinsky \cite{Pinsky 1995}.

\bigskip{}
 As a corollary, we are now in a position to prove our second main
result.
\end{rem*}
\begin{thm}
\label{1} Suppose that $\varphi\in C^{3}(\mathbb{R})$ satisfies
(H.1) and (H.2) for some $c\in\mathbb{R}$, and $\varphi^{(i)}$ (where
$i=0,1,2,3$) have at most polynomial growth. Suppose that $g$ satisfies
(A.1) and (A.2), $g\in C_{b}^{1,3}(\mathbb{R}_{+}\times\mathbb{R}\times\mathbb{R})$,
and $g(t,y,\cdot)$ is symmetric about $0$, that is, $g(t,y,z)=g(t,y,-z)$
for all $t\geq0,y,z\in\mathbb{R}$. Let $(Y_{t},Z_{t})$ be the solution
pair of BSDE (\ref{3.1}). Then the following conclusions hold.

(1) If $\varphi'(x)\ge0$ and and $\varphi'(x)\not\equiv0$ for all
$x>c$, then $sgn(Z_{t})=sgn(B_{t}-c)$ for all $t\geq0$ almost surely.

(2) Similarly, if $\varphi'(x)\ge0$ and $\varphi'(x)\not\equiv0$
for all $x>c$, then $sgn(-Z_{t})=sgn(B_{t}-c)$ for all $t\geq0$
almost surely.
\end{thm}

\begin{proof}
By Theorem \ref{lemma1}, $Z_{t}=w(T-t,B_{t})$, and
\[
w(t,x)=E\left[N_{t}\varphi'(X_{t}^{x})e^{\int_{0}^{t}a_{s,t}ds}I_{\left\{ t<\tau\right\} }\right]
\]
which allows to determine the sign of $w(t,x)$ accordingly.

Note that if $x>c$, then $X_{t}^{x}>c$ on $t<\tau$, \ so, unless
$\varphi'(x)$ equals zero identically for $x>c$, we must have $P\left(N_{t}\varphi'(X_{t}^{x})1_{\left\{ t<\tau\right\} }>0\right)>0$.
Since $N_{t}>0$, thus if $\varphi'\geq0$ and $\varphi'$ does not
vanish identically on $(c,\infty)$, we have $w(t,x)>0$ for $x>c$
and $w(t,x)<0$ for $x<c$, which implies $\textrm{sgn}(Z_{t})=\textrm{sgn}(B_{t}-c)$.

Similarly, if $\varphi'(x)\leq0$ and $\varphi'(x)\not\equiv0$ for
all $x>c$, we have $\textrm{sgn}(Z_{t})=-\textrm{sgn}(B_{t}-c)$.

The proof of Theorem \ref{1} is completed.
\end{proof}
Theorem \ref{1} may be stated as the following ``non-vanishing theorem'',
which is the most useful form in our discussions below.
\begin{thm}
Under the same assumptions on $g$ and $\varphi$ in Theorem \ref{1},
and suppose $(Y_{t},Z_{t})$ is the unique solution of BSDE (\ref{3.1}).
Then $Z\neq0$ with respect to the product measure $dt\otimes dP$.
\end{thm}

\begin{proof}
This is a direct consequence of Theorem \ref{1}, as $\left\{ Z\neq0\right\} =\left\{ B\neq c\right\} $
almost surely, $B\neq c$ almost surely with respect to $dt\otimes dP$.
\end{proof}

While the conditions imposed on the initial data $\varphi$ and the
regularity imposed on the non-linear driver $g(t,y,z)$ in Theorem
\ref{1} are too restrictive in applications, which are needed to
achieve a general result, though these conditions are sufficient but
very often not necessary. Here we do not seek for the best conditions
in particular on $\varphi$, and the approach put forward in Theorem
\ref{1} however also applies to situations where the regularity on
the driver $g(t,y,z)$ is not available. These instances however have
to be treated case by case. In this article we deal with an important
example, the $k$-ignorance model, with details.

The non-linear driver is only Lipschitz continuous in the $k$-ignorance
model, which has a unique solution pair $(Y,Z)$ according to Pardoux-Peng
\cite{peng1}. One however can not apply the non-linear Feynman-Kac
formula directly, as the solution $u(t,x)$ to the corresponding parabolic
equation is only $C^{1+}$, but not $C^{2}$ in the variable $x$
in general. Thus the main effort is to derive a non-linear Feynman-Kac
type formula for this case, and generalize the results in Theorem
\ref{lemma1} to the current example.
\begin{thm}
\label{mo} Let $\varphi\in C^{3}(\mathbb{R})$ satisfying (H.1) and
(H.2) with some constant $c$, such that $\varphi$ and $\varphi'$
have at most polynomial growth, and let $u$ be the unique weak solution
to the non-linear parabolic equation
\begin{equation}
\partial_{t}u=\frac{1}{2}\partial_{xx}^{2}u+k|\partial_{x}u|\label{eq:qpe1-1}
\end{equation}
with the initial condition that
\begin{equation}
u(0,x)=\varphi(x).\label{eq:qpe2-1}
\end{equation}
Then $\partial_{x}u(t,x)$ is H$\ddot{o}$lder continuous in any compact
subset of $(0,\infty)\times\mathbb{R}$, and for every $t>0$ and
$x\in\mathbb{R}$
\begin{equation}
\partial_{x}u(t,x)=E\left[N_{t}\varphi'(B_{t}+x)\cdot1_{\left\{ t<\tau\right\} }\right],\label{eq:main-rep-01}
\end{equation}
where
\[
N_{s}=\exp\left[k\int_{0}^{s}\textrm{sgn}(w(t-r,B_{r}+x))dB_{r}-\frac{k^{2}}{2}s\right]
\]
is a martingale for $0\leq s\leq t$, $w(t,x)=\partial_{x}u(t,x)$
is the unique weak solution to the initial value problem of the parabolic
equation
\begin{equation}
\partial_{t}w=\frac{1}{2}\partial_{xx}^{2}w+k\textrm{sgn}(\partial_{x}u(t,x))\cdot\partial_{x}w\label{eq:pde-qa1}
\end{equation}
with the initial condition that
\begin{equation}
w(0,x)=\varphi'(x),\label{eq:qpe2-1-1}
\end{equation}
and $\tau=\inf\left\{ s\geq0:B_{s}+x=c\right\} $. Moreover $Y_{t}=u(T-t,B_{t})$
and $Z_{t}=w(T-t,B_{t})$ is the unique solution pair to the $k$-ignorance
model.
\end{thm}

\begin{proof}
According to the theory of parabolic equations \cite{Friedman 1964,Ladyzenskaja Solonnikov and Uralceva},
there is a unique weak solution $u(x,t)$ to the problem (\ref{eq:qpe1-1},
\ref{eq:qpe2-1}), and $\partial_{x}u(x,t)$ is H$\ddot{o}$lder continuous
on any compact subset of $(0,T)\times\mathbb{R}^{d}$. According to
Aronson's estimate and Nash-Moser theory (see \cite{Aronson 1968,Nash 1958}
for details), it follows that the linear problem (\ref{eq:pde-qa1},
\ref{eq:qpe2-1-1}) has a unique weak solution which is H$\ddot{o}$lder
continuous in any compact set of $(0,T)\times\mathbb{R}$.

Next we prove that $Y_{t}=u(T-t,B_{t})$ and $Z_{t}=w(T-t,B_{t})$
are the unique solution pair of BSDE (\ref{1.2}). To this end, for
$\varepsilon\geq0$, let $g_{\varepsilon}(z)=k\sqrt{z^{2}+\varepsilon}$.
For $\varepsilon>0$, $g_{\varepsilon}$ is smooth and $|g_{\varepsilon}(z)-g_{0}(z)|\rightarrow0$
as $\varepsilon\rightarrow0$ for every $z\in\mathbb{R}$. Moreover
$g'_{\varepsilon}(z)=k\frac{z}{\sqrt{z^{2}+\varepsilon}}$ so that
$|g'_{\varepsilon}(z)|\leq k$.

The condition that $\varphi$ possesses at most polynomial growth
is sufficient to ensure the existence of uniqueness of a strong solution
$u^{\varepsilon}(t,x)$ to the problem
\begin{equation}
\partial_{t}u^{\varepsilon}(x,t)=\frac{1}{2}\partial_{xx}^{2}u^{\varepsilon}(t,x)+g_{\varepsilon}(\partial_{x}u^{\varepsilon}(t,x))\label{eq:qpe1}
\end{equation}
together with the initial condition that
\begin{equation}
u^{\varepsilon}(x,0)=\varphi(x),\label{eq:qpe2}
\end{equation}
for every $\varepsilon>0$. According to the regularity theory (see
\cite{Ladyzenskaja Solonnikov and Uralceva}) of quasi-linear parabolic
equations, $u^{\varepsilon}\in C^{1,\infty}((0,\infty)\times\mathbb{R})$,
whose space derivative function $w^{\varepsilon}(t,x)=\partial_{x}u^{\varepsilon}(t,x)$
is the unique weak solution to the (linear) parabolic equation
\begin{equation}
\partial_{t}w^{\varepsilon}(t,x)=\frac{1}{2}\partial_{xx}^{2}w^{\varepsilon}(t,x)+g'_{\varepsilon}(\partial_{x}u^{\varepsilon}(t,x))\cdot\partial_{x}w^{\varepsilon}(t,x)\label{w-pqe3}
\end{equation}
subject to the initial value that
\begin{equation}
w^{\varepsilon}(0,x)=\varphi'(x).\label{w-pqe4}
\end{equation}
By standard theory of parabolic equations, $u^{\varepsilon}\rightarrow u$
as $\varepsilon\rightarrow0$, where $u$ is the unique weak solution
to the initial problem of the parabolic equation, that is the case
where $\varepsilon=0$ for the problem (\ref{eq:qpe1}, \ref{eq:qpe2}).

We note that $g'_{\varepsilon}$ are uniformly bounded by $|k|$,
which is crucial in our argument below, according to Nash's continuity
theory (see \cite{Nash 1958}), the solutions $\left\{ w^{\varepsilon}(t,x)\right\} $
are uniformly Hölder continuous in any compact subset of $(0,\infty)\times\mathbb{R}$,
and bounded in $L^{2}([0,T],H_{\textrm{loc}}^{1})$ (where $H^{1}$
is the usual Sobolev space), so that we may extract, if necessary,
a sequence $\varepsilon_{n}\downarrow0$, such that $w^{\varepsilon_{n}}(t,x)$
converges to $w(t,x)$ point-wise, uniform in any compact subset of
$(0,T]\times\mathbb{R}$, and $w^{\varepsilon_{n}}$ converges weakly
to $w$ in $L^{2}([0,T],H_{\textrm{loc}}^{1})$. For every $\varepsilon>0$,
$w^{\varepsilon}$ is a strong solution to (\ref{w-pqe3}) so that,
for every $\rho(x,t)$ with a compact support in $[0,T)\times\mathbb{R}$,
by integration by parts
\[
-\int_{\mathbb{R}}\rho(x,0)\varphi'(x)=-\frac{1}{2}\int_{\mathbb{R}\times[0,T)}\partial_{x}\rho(x,t)\partial_{x}w^{\varepsilon}(t,x)+\int_{\mathbb{R}\times[0,T)}\rho(x,t)g'_{\varepsilon}(\partial_{x}u^{\varepsilon}(t,x))\cdot\partial_{x}w^{\varepsilon}(t,x).
\]
Letting $\varepsilon\rightarrow0$, we therefore obtain that
\[
-\int_{\mathbb{R}}\rho(x,0)\varphi'(x)=-\frac{1}{2}\int_{\mathbb{R}\times[0,T)}\partial_{x}\rho(x,t)\partial_{x}w(t,x)+\int_{\mathbb{R}\times[0,T)}\rho(x,t)k\textrm{sgn}(w(t,x))\cdot\partial_{x}w(t,x)
\]
which implies that $w(x,t)$ is the unique weak solution to the problem
(\ref{eq:pde-qa1}, \ref{eq:qpe2-1-1}).

Since for every $n$, according to It$\hat{o}$'s formula
\begin{equation}
Y_{t}^{n}=\varphi(B_{T})+\int_{t}^{T}g_{\varepsilon_{n}}(Z^{n})ds-\int_{t}^{T}Z_{s}^{n}dB_{s}\label{1.2-1}
\end{equation}
where $Z^{\varepsilon_{n}}=w^{\varepsilon_{n}}(T-\cdot,B_{\cdot})\rightarrow w(T-\cdot,B_{\cdot})$
and $Y_{t}^{n}=u^{\varepsilon_{n}}(T-t,B_{t})\rightarrow u(T-t,B_{t})$
as $n\rightarrow\infty$, and therefore $Y_{t}=u(T-t,B_{t})$ and
$Z_{t}=w(T-t,B_{t})$ are the unique solution pair of BSDE
\[
Y_{t}=\varphi(B_{T})+\int_{t}^{T}k|Z_{s}|ds-\int_{t}^{T}Z_{s}dB_{s}\quad\textrm{ for }0\leq t\leq T.
\]

Since $\varphi\in C^{3}(\mathbb{R})$ with polynomial growth, so that
we may apply Theorem \ref{lemma1} to $u^{\varepsilon}(t,x)$. Thus
for each $\varepsilon>0$, $w^{\varepsilon}(t,c)=0$ for all $t\geq0$,
and
\begin{equation}
w^{\varepsilon}(t,x)=E\left[N_{t}^{\varepsilon}\varphi'(B_{t}+x)\cdot I_{\left\{ t<\tau\right\} }\right]\label{rep-q1}
\end{equation}
where $\tau=\inf\left\{ s\geq0:B_{s}+x=c\right\} $ and
\[
N_{t}^{\varepsilon}=\exp\left[\int_{0}^{t}g'_{\varepsilon}(w^{\varepsilon}(t-s,B_{s}+x))dB_{s}-\frac{1}{2}\int_{0}^{t}|g'_{\varepsilon}(w^{\varepsilon}(t-s,B_{s}+x))|^{2}ds\right].
\]
So by Lebesgue's dominated convergence theorem,
\[
E\left[N_{t}^{\varepsilon_{n}}\varphi'(B_{t}+x)\cdot I_{\left\{ t<\tau\right\} }\right]\rightarrow E\left[N_{t}\varphi'(B_{t}+x)\cdot I_{\left\{ t<\tau\right\} }\right]
\]
as $n\rightarrow\infty$, and we may conclude that
\[
w(t,x)=E\left[N_{t}\varphi'(B_{t}+x)\cdot I_{\left\{ t<\tau\right\} }\right]
\]
which completes the proof.
\end{proof}
\begin{rem}
Let us supply some details we omitted in the proof in applying Nash's
theory and Aronson's estimates to our situation, in proving that,
we may extract a sequence $\varepsilon_{n}\downarrow0$, so that $w^{\varepsilon_{n}}(t,x)\rightarrow w(t,x)$,
where $w$ is the unique weak solution to the problem (\ref{eq:pde-qa1},
\ref{eq:qpe2-1-1}). Let us use the same notations in the previous
proof, but for simplicity $b_{\varepsilon}(t,x)=g'_{\varepsilon}(\partial_{x}u^{\varepsilon}(t,x))$
for every $\varepsilon>0$, and $b_{0}(t,x)=k\textrm{sgn}(\partial_{x}u(t,x))$.
Then $|b_{\varepsilon}(t,x)|\leq|k|$, a bound dependent of $\varepsilon$.
Then, according to Nash \cite{Nash 1958} and Aronson \cite{Aronson 1968},
the fundamental solution $p_{\varepsilon}(s,x,t,y)$ to the parabolic
equation
\begin{equation}
\left(\partial_{t}-\frac{1}{2}\Delta-b_{\varepsilon}(t,x)\right)v=0\textrm{ in }(0,\infty)\times\mathbb{R}\label{eq:par-s1}
\end{equation}
is jointly $\alpha$-Hölder continuous for some $\alpha$ depending
only on $|k|$ (see page 328, Friedman \cite{Friedman 1964} or Nash
\cite{Nash 1958}), $p_{\varepsilon}(s,x,t,y)$ satisfying a Gaussian
lower and upper bounds uniformly in $\varepsilon\geq0$, for $0\leq s<t\leq T$,
$x,y\in\mathbb{R}$. This implies that $p_{\varepsilon}(s,x,t,y)$
is $\alpha$-Hölder continuous in all its arguments, where $\alpha$
and Hölder constant depend only on $|k|$ but independent of $\varepsilon>0$.
According to Aronson \cite{Aronson 1968}, the unique weak solution
$w^{\varepsilon}(t,x)$ to (\ref{w-pqe3}, \ref{w-pqe4}) (for all
$\varepsilon\geq0$) has the representation
\begin{equation}
w^{\varepsilon}(t,x)=\int_{\mathbb{R}}p_{\varepsilon}(0,x,t,y)\varphi'(y)dy.\label{eq:par-s2}
\end{equation}
Note also that $\left\{ p_{\varepsilon}(s,x,t,y)\right\} $ is a family
of equi-continuous functions on any compact set of $0\leq s<t$, $x,y\in\mathbb{R}$.
Of course $p_{\varepsilon}(s,x,t,y)$ depends on the solution $w^{\varepsilon}$,
which is a necessary feature for non-linear PDEs, but this does not
cause any difficulty for us. Hence, by extracting a sequence, we can
assume that $p_{\varepsilon_{n}}(s,x,t,y)$ converges on $\{0\leq s<t\}\times\mathbb{R}^{2}$
to $p(s,x,t,y)$, uniformly on any its compact subset, and therefore
$w^{\varepsilon_{n}}(t,x)$ converges to $w(t,x)$ on $(0,\infty)\times\mathbb{R}$,
and uniformly on any its compact subset, so that $p(s,x,t,y)=p_{0}(s,x,t,y)$
and $w(t,x)=w^{0}(t,x)$.

\end{rem}

\begin{cor}
Suppose that $\varphi\in C^{3}(\mathbb{R})$ satisfies (H.1) and (H.2)
with some constant $c$, such that $\varphi$ and $\varphi'$ have
at most polynomial growth, and suppose that $(Y_{t},Z_{t})$ is the
unique solution of BSDE:
\begin{equation}
Y_{t}=\varphi(B_{T})+\int_{t}^{T}k|Z_{s}|ds-\int_{t}^{T}Z_{s}dB_{s},\label{3.2}
\end{equation}
where $k$ is a real constant.
\begin{description}
\item [{{(1)}}] If $\varphi'\geq0$ and $\varphi'\not\equiv0$ on $(c,\infty)$,
then
\[
sgn(Z_{t})=sgn(B_{t}-c),\;\;t\ge0.
\]
\item [{{(2)}}] If $\varphi'\leq0$ and $\varphi'\not\equiv0$ on $(c,\infty)$,
then
\[
sgn(-Z_{t})=sgn(B_{t}-c),\;\;t\ge0.
\]
\end{description}
\end{cor}

\begin{proof}
The conclusions follows from Theorem \ref{mo} follow now immediately.
\end{proof}

\section{Explicit solutions for some BSDEs}

\label{section 3}

Firstly, Theorem \ref{mo} allows us to work out the explicit solution
of BSDE(\ref{3.2}). To this end, we recall the joint distribution
$P(B_{t}\in dx,L_{t}^{\ell}\in dy)$ of $B_{t}$ and its local time
$L_{t}^{\ell}$ with respect to $\ell$ given by
\begin{equation}
\begin{split}P(B_{t}\in dx,L_{t}^{\ell}\in dy) & =\frac{1}{\sqrt{2\pi t^{3}}}(y+|x-\ell|+|\ell|)\exp\bigg\{\frac{-(y+|x-\ell|+|\ell|)^{2}}{2t}\bigg\}1_{\{y>0\}}dxdy\\
 & \quad+\frac{1}{\sqrt{2\pi t}}\left[\exp\left\{ -\frac{x^{2}}{2t}\right\} -\exp\left\{ -\frac{(|x-\ell|+|\ell|)^{2}}{2t}\right\} \right]1_{\{y=0\}}dxdy,
\end{split}
\label{joint}
\end{equation}
see \cite{bor} for example.
\begin{thm}
\label{2} Suppose $\varphi\in C^{1}(\mathbb{R})$ satisfying (H.1)
and (H.2) such that $\varphi(B_{T})$ and $\varphi'(B_{T})$ are square
integrable. Then the unique solution of BSDE (\ref{3.2}) is given
by
\begin{equation}
Y_{t}=H(B_{t}),\quad Z_{t}=\partial_{h}H(B_{t}),\label{Y_t-Z_t}
\end{equation}
where $H$ is defined in the following.
\begin{description}
\item [{{(i)}}] If $\varphi'\geq0$ and $\varphi'\not\equiv0$ on $(c,\infty)$,
then
\[
H(h)=e^{-\frac{1}{2}k^{2}(T-t)}\bigg\{\int_{\mathbb{R}}\int_{y\geq0}\varphi(x+h)e^{k|x-c+h|-k|c-h|-ky}P(B_{T-t}\in dx,L_{T-t}^{c-h}\in dy)\bigg\}.
\]
\item [{{(ii)}}] If $\varphi'\leq0$ and $\varphi'\not\equiv0$ on $(c,\infty)$,
then
\begin{equation}
H(h)=e^{-\frac{1}{2}k^{2}(T-t)}\bigg\{\int_{\mathbb{R}}\int_{y\geq0}\varphi(x+h)e^{-k|x-c+h|+k|c-h|+ky}P(B_{T-t}\in dx,L_{T-t}^{c-h}\in dy)\bigg\}.\label{h_z2}
\end{equation}
\end{description}
\end{thm}

\begin{proof}
(i) Since $\varphi'\geq0$ and $\varphi'\not\equiv0$ on $(c,\infty)$,
by Theorem \ref{mo}, $\textrm{sgn}(Z_{t})=\textrm{sgn}(B_{t}-c)$,
BSDE (\ref{3.2}) can be rewritten as a linear BSDE in $Z$
\[
Y_{t}=\varphi(B_{T})+k\int_{t}^{T}\textrm{sgn}(B_{t}-c)Z_{s}ds-\int_{t}^{T}Z_{s}dB_{s}
\]
whose solution is given by
\begin{equation}
Y_{t}=E\left[\varphi(B_{T})\cdot e^{-\frac{1}{2}k^{2}(T-t)+k\int_{t}^{T}\textrm{sgn}(B_{s}-c)dB_{s}}\bigg|\mathcal{F}_{t}\right].\label{e-1-1}
\end{equation}
By Tanaka's Formula,

\[
\int_{0}^{T}\textrm{sgn}(B_{s}-c)dB_{s}=|B_{T}-c|-|c|-L_{T}^{c}.
\]
Combine with (\ref{e-1-1}), we have
\begin{equation}
\begin{split}Y_{0} & =E_{P}\left[\varphi(B_{T})e^{-\frac{1}{2}k^{2}T+k\int_{0}^{T}\textrm{sgn}(B_{s}-c)dB_{s}}\right]\\
 & =E_{P}\left[\varphi(B_{T})\cdot e^{-\frac{1}{2}k^{2}T+k\left(|B_{T}-c|-|c|-L_{T}^{c}\right)}\right]\\
 & =e^{-\frac{1}{2}k^{2}T}\int_{\mathbb{R}}\int_{y\geq0}\varphi(x)\exp\left\{ k|x-c|-k|c|-ky\right\} P(B_{T}\in dx,L_{T}^{c}\in dy).
\end{split}
\label{Y0}
\end{equation}
and
\begin{eqnarray*}
Y_{t} & = & e^{-\frac{1}{2}k^{2}(T-t)}\int_{\mathbb{R}}\int_{y\geq0}\varphi(x+h)\exp\left\{ k|x-c+h|-k|c-h|-ky\right\} P(B_{T-t}\in dx,L_{T-t}^{c-h}\in dy)|_{h=B_{t}}\\
 & = & H(B_{t}).
\end{eqnarray*}
According to Theorem \ref{mo}, $Z_{t}=\partial_{h}H(B_{t})$.

(ii) Similarly, if $\varphi'\leq0$ and $\varphi'\not\equiv0$ on
$(c,\infty)$, by Theorem \ref{mo} we have $\textrm{sgn}(Z_{t})=-\textrm{sgn}(B_{t}-c)$.
The rest can be proved in a similar manner as (i). The proof is completed.
\end{proof}
Another application of Theorem \ref{1} is to prove that the following
different PDEs have the same solution under some assumptions on initial
value $\varphi.$

Let $v$ and $u$ be the solutions of the following initial value
problems of parabolic equations
\begin{equation}
\begin{cases}
 & \partial_{t}v(t,x)=\frac{1}{2}\partial_{xx}^{2}v(t,x)-k\cdot\textrm{sgn}(x-c)\partial_{x}v(t,x)\\
 & v(0,x)=\varphi(x),
\end{cases}\label{v1}
\end{equation}
and
\begin{equation}
\begin{cases}
 & \partial_{t}u=\frac{1}{2}\partial_{xx}^{2}u+\underset{|m|\leq k}{\min}(m\;\partial_{x}u)\\
 & u(0,x)=\varphi(x),
\end{cases}\label{v2}
\end{equation}
respectively.

Shreve \cite{shr} show that for $\varphi(x)=x^{2}$, then $v(t,x)=u(t,x)$
for $c=0$. The following Corollary implies that both solutions may
be same for all symmetric functions $\varphi$ satisfying (H.1) and
$\varphi$ is increasing on $(c,\infty)$.
\begin{cor}
Assume that $\varphi$ satisfies (H.1), (H.2), moreover, $\varphi'\geq0$
on $(c,\infty)$ and $\varphi'\not\equiv0$ on $(c,\infty).$ Then
PDE (\ref{v1}) and PDE (\ref{v2}) have the same solution ,that is,
$v(t,x)=u(t,x)$ for all $(t,x)\in\mathbb{R}_{+}\times\mathbb{R}$.
\end{cor}

\begin{proof}
It is obvious that PDE (\ref{v2}) is equivalent to the following
PDE:
\[
\begin{cases}
 & \partial_{t}u=\frac{1}{2}\partial_{xx}^{2}u-k\cdot|\partial_{x}u|\\
 & u(0,x)=\varphi(x).
\end{cases}
\]
As we have proven in Theorem \ref{mo}, $Y_{t}=u(T-t,B_{t})$, $Z_{t}=\partial_{x}u(T-t,B_{t})$
is the unique solution pair of the following BSDE:
\[
Y_{t}=\varphi(B_{T})-k\int_{t}^{T}|Z_{s}|ds-\int_{t}^{T}Z_{s}dB_{s}
\]
By Theorem \ref{mo}, when $\varphi$ satisfies (H.1) and (H.2), $\varphi'\not\equiv0$
on $(c,\infty)$, $\textrm{sgn}(Z_{t})=\textrm{sgn}(B_{t}-c)$, which
means $\textrm{sgn}(\partial_{x}u(t,x))=\textrm{sgn}(x-c)$. Thus
the claim follows immediately.
\end{proof}

\section{Applications}

\label{section-4}

In this section, we give several examples to show how to get the explicit
solution of BSDE for the cases where $\varphi(x)=x^{2}$ and $\varphi(x)=I_{\{a\leq x\leq b\}}.$
\begin{example}
\label{example-1} The explicit solution $(Y_{t},Z_{t})$ of BSDE
\begin{equation}
Y_{t}=B_{T}^{2}+k\int_{t}^{T}|Z_{s}|ds-\int_{t}^{T}Z_{s}dB_{s}.\label{4.1}
\end{equation}
is given by the following formulas:
\begin{equation}
\begin{split}Y_{t}= & \frac{1}{2k^{2}}+\sqrt{\frac{T-t}{2\pi}}\left[|B_{t}|+k(T-t)+\frac{1}{k}\right]\exp\bigg\{-\frac{[|B_{t}|+k(T-t)]^{2}}{2(T-t)}\bigg\}\\
 & 
+\bigg\{[|B_{t}|+k(T-t)]^{2}+(T-t)-\frac{1}{2k^{2}}\bigg\}\Phi\left(\frac{|B_{t}|+k(T-t)}{\sqrt{T-t}}\right)\\
 & +e^{-2k|B_{t}|}(|B_{t}|+T-t-\frac{1}{2k^{2}})\Phi\left(-\frac{|B_{t}|-k(T-t)}{\sqrt{T-t}}\right)
\end{split}
\label{x2-yt}
\end{equation}
and
\begin{equation}
\begin{split}Z_{t}= & \sqrt{\frac{T-t}{2\pi}}\cdot\textrm{sgn}(B_{t})\cdot\exp\left\{ -\frac{[|B_{t}|+k(T-t)]^{2}}{2(T-t)}\right\} \cdot\left\{ 1+\left[|B_{t}|+k(T-t)+\frac{1}{k}\right]\cdot\left[-\frac{[|B_{t}|+k(T-t)]}{(T-t)}\right]\right\} \\
 & 
+2\textrm{sgn}(B_{t})\cdot[|B_{t}|+k(T-t)]\cdot\Phi\left(\frac{|B_{t}|+k(T-t)}{\sqrt{T-t}}\right)\\
 & +\left\{ [|B_{t}|+k(T-t)]^{2}-k(T-t)-\frac{1}{2k^{2}}\right\} \frac{\textrm{sgn}(B_{t})}{\sqrt{2\pi(T-t)}}\exp\left\{ -\frac{[|B_{t}|+k(T-t)]^{2}}{2(T-t)}\right\} \\
 & +e^{-2k|B_{t}|}\textrm{sgn}(B_{t})\Phi\left(-\frac{|B_{t}|-k(T-t)}{\sqrt{T-t}}\right)\lt[-2k\lt(|B_{t}|+T-t-\frac{1}{2k^{2}}\rt)+1\rt]\\
 & -e^{-2k|B_{t}|}\left(|B_{t}|+T-t-\frac{1}{2k^{2}}\right)\frac{\textrm{sgn}(B_{t})}{\sqrt{2\pi(T-t)}}\exp\left\{ -\frac{[|B_{t}|-k(T-t)]^{2}}{2(T-t)}\right\} .
\end{split}
\label{x2-zt}
\end{equation}
Here and in the sequel, $\Phi$ is the standard normal cdf.
\end{example}

\begin{proof}
\ By Theorem \ref{2},
\[
Y_{t}=e^{-\frac{1}{2}k^{2}(T-t)}\bigg\{\int_{\mathbb{R}}\int_{y\geq0}(x+h)^{2}e^{k|x+h|-k|h|-ky}P(B_{T-t}\in dx,L_{T-t}^{-h}\in dy)\bigg\}\bigg|_{h=B_{t}}.
\]
By an elementary calculation, we have
\begin{equation}
\begin{split}Y_{t}= & \frac{1}{2k^{2}}+\sqrt{\frac{T-t}{2\pi}}\left[|B_{t}|+k(T-t)+\frac{1}{k}\right]\exp\bigg\{-\frac{[|B_{t}|+k(T-t)]^{2}}{2(T-t)}\bigg\}\\
 & 
+\bigg\{[|B_{t}|+k(T-t)]^{2}+(T-t)-\frac{1}{2k^{2}}\bigg\}\Phi\left(\frac{|B_{t}|+k(T-t)}{\sqrt{T-t}}\right)\\
 & +e^{-2k|B_{t}|}(|B_{t}|+T-t-\frac{1}{2k^{2}})\Phi\left(-\frac{|B_{t}|-k(T-t)}{\sqrt{T-t}}\right)
\end{split}
\label{x2-yt}
\end{equation}
Therefore, by Theorem \ref{mo}, we have $Z_{t}=\partial_{h}H(B_{t})$,
and therefore we obtain $Z_{t}$ as equation (\ref{x2-zt}). The proof
is completed.
\end{proof}
It is interesting that our result may be also used to get the solution
of the following PDE, its application can be found in \cite{shreve}
for $k=-1$.
\begin{example}
The unique weak solution to the initial problem of the parabolic equation
\[
\begin{cases}
 & \partial_{t}u(t,x)=\frac{1}{2}\partial_{xx}^{2}u(t,x)+k\cdot\textrm{sgn}(x)\partial_{x}u(t,x),\\
 & \lim\limits _{t\to0^{+}}u(t,x)=x^{2},
\end{cases}
\]
is given as the following
\[
\begin{split}u(t,x)= & \frac{1}{2k^{2}}+\sqrt{\frac{t}{2\pi}}(|x|+kt+\frac{1}{k})\exp\bigg\{-\frac{(|x|+kt)^{2}}{2t}\bigg\}\\
 & 
+\bigg\{(|x|+kt)^{2}+t-\frac{1}{2k^{2}}\bigg\}\Phi\left(\frac{|x|+kt}{\sqrt{t}}\right)\\
 & +e^{-2k|x|}(|x|+T-t-\frac{1}{2k^{2}})\Phi\left(-\frac{|x|-kt}{\sqrt{t}}\right).
\end{split}
\]
The explicit solution agrees with the result of \cite{shreve} when
$k=-1$.
\end{example}

\begin{proof}
By Theorem \ref{mo}, $Y_{t}=u(T-t,B_{t}),\ Z_{t}=\partial_{x}u(T-t,B_{t})$
are the unique solution pair to the BSDE
\[
Y_{t}=B_{T}^{2}+k\int_{t}^{T}|Z_{s}|ds-\int_{t}^{T}Z_{s}dB_{s}.
\]
By Theorem \ref{1}, $\textrm{sgn}(B_{t})=\textrm{sgn}(Z_{t})$ for
$\varphi(x)=x^{2}$. Hence the expression for $u(t,x)$ follows from
(\ref{x2-yt}) immediately.
\end{proof}
\begin{example}
\label{3} We now calculate the solution of the following BSDE:
\begin{equation}
Y_{t}=I_{[a\leq B_{T}\leq b]}+\int_{t}^{T}k|Z_{s}|ds-\int_{t}^{T}Z_{s}dB_{s}.\label{indicator}
\end{equation}
\begin{description}
\item [{{(1)}}] For any $a,b\in(-\infty,\infty),$ set $c=\frac{a+b}{2},$

\begin{equation}
Y_{t}=\Phi\left(-\frac{|B_{t}-c|-k(T-t)-\frac{b-a}{2}}{\sqrt{T-t}}\right)-e^{-k(b-a)}\Phi\left(-\frac{|B_{t}-c|-k(T-t)+\frac{b-a}{2}}{\sqrt{T-t}}\right),\label{Y_t}
\end{equation}
and
\begin{equation}
Z_{t}=\frac{-\textrm{sgn}(B_{t}-c)}{\sqrt{2\pi(T-t)}}\left\{ e^{-\frac{[|B_{t}-c|-k(T-t)-\frac{b-a}{2}]^{2}}{2(T-t)}}-e^{-k(b-a)}e^{-\frac{[|B_{t}-c|-k(T-t)+\frac{b-a}{2}]^{2}}{2(T-t)}}\right\} .\label{Z_t}
\end{equation}

\item [{{(2)}}] If $a=-\infty$, then for any $b<\infty$,
\[
Y_{t}=\Phi\left(-\frac{B_{t}-k(T-t)-b}{\sqrt{T-t}}\right),\ \ Z_{t}=-\frac{1}{\sqrt{2\pi(T-t)}}e^{-\frac{[B_{t}-k(T-t)-b]^{2}}{2(T-t)}}.
\]
\item [{{(3)}}] If $b=+\infty$, then for any $a>-\infty,$ we have
\[
Y_{t}=\Phi\left(\frac{B_{t}+k(T-t)-a}{\sqrt{T-t}}\right),\ \ Z_{t}=\frac{1}{\sqrt{2\pi(T-t)}}e^{-\frac{[B_{t}+k(T-t)-a]^{2}}{2(T-t)}}.
\]
\end{description}
\end{example}

\begin{proof}
(1) For any $\varepsilon>0$, set
\[
\varphi_{\varepsilon}(x)=E[I_{[a,b)}(x+\sqrt{\varepsilon}\xi)]=\int_{-\infty}^{\infty}I_{[a,b]}(v)\frac{1}{\sqrt{2\pi\varepsilon}}\exp\left[-\frac{(v-x)^{2}}{2\varepsilon}\right]dv
\]
where $\xi$ is a standard normal distribution under probability measure
$P$. Then $\varphi_{\varepsilon}\in C^{\infty}(\mathbb{R})$ and
$\varphi_{\varepsilon}(x)\rightarrow I_{[a,b)}(x)$ as $\varepsilon\rightarrow0$.

Consider the following BSDE
\[
Y_{t}^{\varepsilon}=\varphi_{\varepsilon}(B_{T})+\int_{t}^{T}k|Z_{s}|ds-\int_{t}^{T}Z_{s}dB_{s}.
\]
By Theorem \ref{2},
\begin{eqnarray*}
Y_{t}^{\varepsilon} & = & e^{-\frac{1}{2}k^{2}(T-t)}\bigg\{\int_{\mathbb{R}}\int_{y\geq0}\varphi_{\varepsilon}(x+h)e^{k|x-c+h|-k|c-h|-ky}P(B_{T-t}\in dx,L_{T-t}^{c-h}\in dy)\bigg\}\bigg|_{h=B_{t}}\\
 & = & e^{-\frac{1}{2}k^{2}(T-t)}\left[\int_{-\infty}^{\infty}\varphi_{\varepsilon}(x+h)\cdot f(c,h,x)dx\right]\bigg|_{h=B_{t}}\\
 & =: & H_{\varepsilon}(h)|_{h=B_{t}},
\end{eqnarray*}
where
\[
f(c,h,x)=\int_{0}^{\infty}\exp\{-k|x-c+h|+k|c-h|+ky\}f_{1}(c,h,x,y)dy+f_{2}(c,h,x),
\]
where
\[
f_{1}(c,h,x,y)=\frac{y+|x-(c-h)|+|c-h|}{\sqrt{2\pi(T-t)^{3}}}\exp\left[-\frac{(y+|x-(c-h)|+|c-h|)^{2}}{2(T-t)}\right],
\]
and
\[
f_{2}(c,h,x)=\frac{e^{-k|x-c+h|+k|c-h|}}{\sqrt{2\pi(T-t)}}\left\{ e^{-\frac{x^{2}}{2(T-t)}}-\exp\left[-\frac{(|x-(c-h)|+|c-h|)^{2}}{2(T-t)}\right]\right\} .
\]

Now we prove that
\[
H_{\varepsilon}(h)\to H(h)\quad\textrm{ as }\ \varepsilon\to0,
\]
where
\[
H(h)=\Phi\left(-\frac{|h-c|-k(T-t)-\frac{b-a}{2}}{\sqrt{T-t}}\right)-e^{-k(b-a)}\Phi\left(-\frac{|h-c|-k(T-t)+\frac{b-a}{2}}{\sqrt{T-t}}\right).
\]
Actually we have
\begin{eqnarray*}
H(h) & = & e^{-\frac{1}{2}k^{2}(T-t)}\bigg\{\int_{\mathbb{R}}\int_{y\geq0}I_{[a,b)}(x+h)e^{k|x-c+h|-k|c-h|-ky}P(B_{T-t}\in dx,L_{T-t}^{c-h}\in dy)\bigg\}\\
 & = & e^{-\frac{1}{2}k^{2}(T-t)}\left[\int_{-\infty}^{\infty}I_{[a,b)}(x+h)\cdot f(c,h,x)dx\right].
\end{eqnarray*}
%
By Lebesgue's Dominated convergence theorem, we have $H_{\varepsilon}(h)$ converges to $H(h)$ as $\varepsilon\to0,$ which means $H_{\varepsilon}(B_t)$ converges to $H(B_t)$ almost surely. Therefore, we have
\[
Y_{t}=H(B_{t})=\Phi\left(-\frac{|B_{t}-c|-k(T-t)-\frac{b-a}{2}}{\sqrt{T-t}}\right)-e^{-k(b-a)}\Phi\left(-\frac{|B_{t}-c|-k(T-t)+\frac{b-a}{2}}{\sqrt{T-t}}\right).
\]

By Corollary 4.1 in El Karoui, Peng and Quenez \cite{peng1}, we have
$Z_{t}=\partial_{h}H(B_{t}).$ So we get $Z_{t}$ given by (\ref{Z_t}).
From which we can see

\[
\textrm{sgn}(Z_{t})=-\textrm{sgn}(B_{t}-c),
\]
which means Theorem \ref{1} also holds for indicator function.

We now prove (2). In fact, since $a=-\infty$, then for any $b\in\mathbb{R}$,
$c=\frac{a+b}{2}=-\infty$. By \ Theorem \ref{mo}, $\textrm{sgn}(Z_{t})=-\textrm{sgn}(B_{t}-c)=-1$,
which implies that $Z_{t}<0$ for \ $t\in[0,T],$ and BSDE(\ref{indicator})
is a linear BSDE:
\[
Y_{t}=I_{[B_{T}\leq b]}-\int_{t}^{T}kZ_{s}ds-\int_{t}^{T}Z_{s}dB_{s}.
\]
By solving the linear BSDE, we obtain (2).

Similarly, we may deduce (3), and we omit the details.
\end{proof}
\begin{rem}
The BSDE (\ref{3.1}) is associated with the parabolic equation
\begin{equation}
\frac{\partial u}{\partial t}=\frac{1}{2}\Delta u+g(t,u,\nabla u),\quad u(0,x)=\varphi(x).
\end{equation}
For the study of the sign of $Z_{t}$, it is actually equivalent to
the study the nodal set of $u_{x}$. It has a connection to the work
of Qian and Xu (2018). For more details, see \cite{qian-xu}.
\end{rem}

We plot one sample path of Brownian motion $B_{t}$ and the solution
$Z_{t}$ of Example \ref{3} in the Figure \ref{fig_bm_z}, in which
the blue line is $B_{t}$ and the red is $Z_{t}$. We can see the
relationship of the sign between $B_{t}-c$ and $Z_{t}$ intuitively
in this figure.

\begin{figure}[!htbp]
\begin{centering}
\includegraphics[width=11cm,height=7cm]{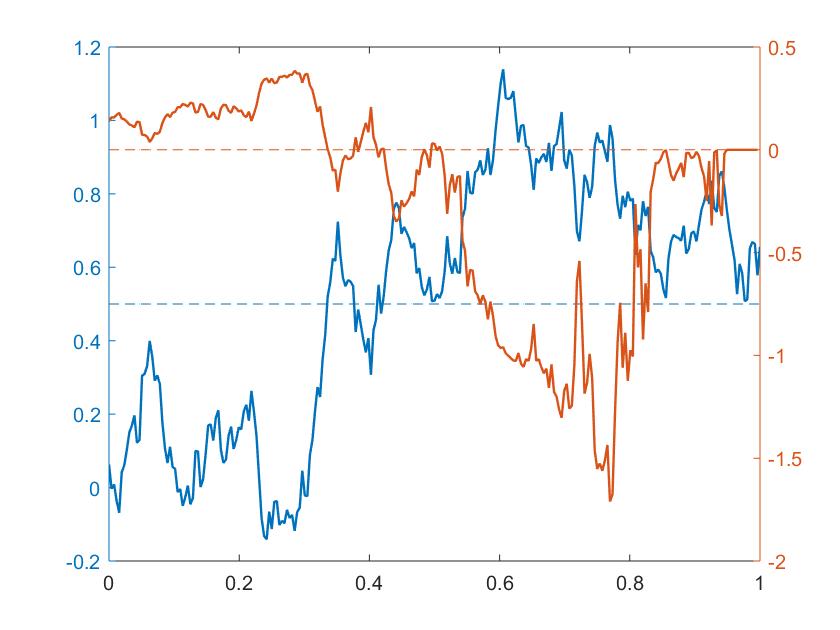}
\par\end{centering}
\caption{Brownian motion $B_{t}$ (blue) and solution $Z_{t}$ (red) ($a=0,\ b=1,\ k=0.1,\ T=1$)}
\label{fig_bm_z}
\end{figure}

\newpage{}

\section{Robust prices in incomplete markets}

\label{section 5}

The Black-Scholes model studied by Black and Scholes (1973), Merton
(1973,1991) is the most celebrated example of option pricing and hedging
in a complete market using no-arbitrage theory and martingale methods.
According to this theory, when a stock obeys the geometric Brownian
motion
\begin{equation}
dS_{t}=\mu S_{t}dt+\sigma S_{t}dB_{t},\ \ S_{0}=1,\label{eq:bs1}
\end{equation}
there exists a unique risk neutral martingale measure $Q$ such that
the price of the contingent claim $\xi$ at time $T$ was given by
$E_{Q}[\xi e^{-rT}]$, where
\begin{equation}
\left.\frac{dQ}{dP}\right|_{\mathcal{F}_{T}}=e^{\int_{0}^{T}(\frac{\mu-r}{\sigma})dB_{s}-\frac{1}{2}\int_{0}^{T}(\frac{\mu-r}{\sigma})^{2}ds},\label{eq:mart-q1}
\end{equation}
where $r$ is the interest rate of a bond. Therefore, for $\xi=I_{(a\leq S_{T}\leq b)}$,
the price of the contingent claim $\xi$ is given by
\begin{equation}
E_{Q}\left(\xi e^{-rT}\right)=e^{-rT}\left[\Phi\left(\frac{\ln b-(2\mu-r-0.5\sigma^{2})T}{\sigma\sqrt{T}}\right)-\Phi\left(\frac{\ln a-(2\mu-r-0.5\sigma^{2})T}{\sigma\sqrt{T}}\right)\right].\label{eq:mart-01}
\end{equation}
In an incomplete market, the incompleteness of the market usually
gives rise to infinitely many martingale measures, therefore upper
and lower pricing was studied by El Karoui and Quenez (1995) \cite{el},
EL Karoui and Peng (1997) \cite{peng2}. They use the min-max pricing
to show that, the pricing of an insurance or contingent claim equals
the maximal (minimal) expectations with respect to a set of martingale
measures. Chen and Epstein (2002) studied the ambiguity pricing under
a set of special measures $\mathcal{P}$, where
\begin{equation}
\mathcal{P}=\left\{ Q:\left.\frac{dQ}{dP}\right|_{\mathcal{F}_{t}}=\exp\left[\int_{0}^{t}\theta_{s}dB_{s}-\frac{1}{2}\int_{0}^{t}\theta_{s}^{2}ds\right],\ \ \underset{s\in[0,T]}{\sup}|\theta_{s}|\leq k\right\} .\label{eq:fam-aq1}
\end{equation}

Set
\begin{equation}
y_{t}=\textrm{ess}\underset{Q\in\mathcal{P}}{\inf}E_{Q}[\xi|\mathcal{F}_{t}]\textrm{ and }Y_{t}=\textrm{ess}\underset{Q\in\mathcal{P}}{\sup}E_{Q}[\xi|\mathcal{F}_{t}].\label{eq:fam-aq2}
\end{equation}
It is known that $y_{t}$ and $Y_{t}$ are respectively the minimum
and maximum price of $\xi$ in an incomplete market. Chen and Epstein
(2002)\cite{chen2} have shown that there exists an adapted $z_{t}$
such that $Y_{t}$ and $y_{t}$ have the following representations:
\begin{equation}
y_{t}=\xi-k\int_{t}^{T}|z_{s}|ds-\int_{t}^{T}z_{s}dB_{s}\label{eq:rep-aaq1}
\end{equation}
and
\begin{equation}
Y_{t}=\xi+k\int_{t}^{T}|z_{s}|ds-\int_{t}^{T}z_{s}dB_{s}.\label{eq:rep-aaq2}
\end{equation}

By using our results in the previous sections, we may give the explicit
representation of the wealth $Y_{t}$ when the stock price $S_{t}$
obeys the geometric Brown motion
\begin{equation}
S_{t}=\exp\left[(\mu-\frac{1}{2}\sigma^{2})t+\sigma B_{t}\right]\label{eq:bs-geom1}
\end{equation}
and $\xi=I_{(a\leq S_{T}\leq b)}$.

In fact, when $\xi=I_{(a\leq S_{T}\leq b)}$, that is,
\begin{equation}
\xi=I_{\left\{ \frac{\ln a-(\mu-0.5\sigma^{2})T}{\sigma}\leq B_{T}\leq\frac{\ln b-(\mu-0.5\sigma^{2})T}{\sigma}\right\} }.\label{eq:term-s1}
\end{equation}
According to the calculation in Example \ref{3}, with $c=\frac{\ln(ab)}{2\sigma}-\frac{(\mu-0.5\sigma^{2})T}{\sigma}$,
we have the upper pricing which is given by
\begin{equation}
Y_{t}=\Phi\left(-\frac{|B_{t}-c|-k(T-t)-\frac{\ln{(b/a)}}{2\sigma}}{\sqrt{T-t}}\right)-e^{-k\frac{\ln{(b/a)}}{\sigma}}\Phi\left(-\frac{|B_{t}-c|-k(T-t)+\frac{\ln{(b/a)}}{2\sigma}}{\sqrt{T-t}}\right)\label{eq:upper-pr1}
\end{equation}
and the lower pricing is given as
\begin{equation}
y_{t}=\Phi\left(-\frac{|B_{t}-c|+k(T-t)-\frac{\ln{(b/a)}}{2\sigma}}{\sqrt{T-t}}\right)-e^{k\frac{\ln{(b/a)}}{\sigma}}\Phi\left(-\frac{|B_{t}-c|+k(T-t)+\frac{\ln{(b/a)}}{2\sigma}}{\sqrt{T-t}}\right).\label{eq:lower-pr1}
\end{equation}
In particular, let $t=0$, then
\begin{equation}
Y_{0}=\Phi\left(-\frac{|c|-kT-\frac{\ln{(b/a)}}{2\sigma}}{\sqrt{T}}\right)-e^{-k\frac{\ln{(b/a)}}{\sigma}}\Phi\left(-\frac{|c|-kT+\frac{\ln{(b/a)}}{2\sigma}}{\sqrt{T}}\right)\label{eq:upper-pr3}
\end{equation}
and
\begin{equation}
y_{0}=\Phi\left(-\frac{|c|+kT-\frac{\ln{(b/a)}}{2\sigma}}{\sqrt{T}}\right)-e^{k\frac{\ln{(b/a)}}{\sigma}}\Phi\left(-\frac{|c|+kT+\frac{\ln{(b/a)}}{2\sigma}}{\sqrt{T}}\right).\label{eq:lower-pr3}
\end{equation}

\end{document}